\documentclass[abstract=on, 11pt, a4paper]{scrartcl}

\usepackage[T1]{fontenc}

\usepackage{amsmath, amssymb, amsthm}
\usepackage{mathtools}
\usepackage{latexsym}

\usepackage{amsbsy}

\usepackage[automark, headsepline, footsepline, plainfootsepline]{scrlayer-scrpage}

\clearscrheadfoot

\clearscrplain
\ofoot[\pagemark]{\pagemark}
\ohead{\headmark}
\automark{section}

\usepackage[left=3cm, right=2.5cm, top=2.5cm, bottom=3.5cm]{geometry}

\usepackage{enumitem}

\usepackage{xcolor}

\usepackage[most]{tcolorbox}
\tcbuselibrary{listings,theorems}

\usepackage[singlespacing]{setspace}

\usepackage{tikz}
\usetikzlibrary{arrows,matrix,trees}
\usetikzlibrary{circuits}

\usepackage{graphicx}
\usepackage{pdfpages}

\usepackage[super]{nth}

\newcounter{countercheck}[subsection]

\theoremstyle{plain}
\swapnumbers
\newtheorem{theorem}[countercheck]{Theorem}
\newtheorem{proposition}[countercheck]{Proposition}
\newtheorem{lemma}[countercheck]{Lemma}
\newtheorem{corollary}[countercheck]{Corollary}

\theoremstyle{definition}

\newtheorem{convention}[countercheck]{Convention}

\theoremstyle{remark}

\renewcommand{\phi}{\varphi}
\renewcommand{\epsilon}{\varepsilon}
\newcommand{\NNN}{\mathbb{N}_0}
\newcommand{\NN}{\mathbb{N}}

\newcommand{\Cp}{\mathcal{C}_{\epsilon}}
\newcommand{\Cm}{\mathcal{C}_{-\epsilon}}
\newcommand{\C}{\mathcal{C}}
\renewcommand{\P}{\mathcal{P}}
\newcommand{\Cbf}{\textbf{\textup{C}}}

\DeclareMathOperator{\proj}{proj}
\DeclareMathOperator{\Opp}{Opp}
\DeclareMathOperator{\co}{\textup{(co)}}

\numberwithin{equation}{section} 
\title{On isometries of twin buildings}
\author{Sebastian \textit{Bischof}$^1$\footnote{email: sebastian.bischof@math.uni-giessen.de} \and Anton \textit{Chosson}$^2$\footnote{email: anton.chosson@u-psud.fr} \and Bernhard \textit{M\"uhlherr}$^1$\footnote{email: bernhard.m.muehlherr@math.uni-giessen.de} \\ \\
	$^1$ Mathematisches Institut, Arndtstra\ss e 2, 35392 Gie\ss en, Germany \\
	$^2$ IUT d'Orsay Universit\'e Paris-Sud, F-91405 Orsay Cedex, France}

\date{\today}

\begin{document}

\maketitle

\section{Introduction}

Twin buildings were introduced by Ronan and Tits in the late
$1980$s. Their definition was motivated by the theory of Kac-Moody groups
over fields. Each such group acts naturally on a pair of buildings
and the action preserves an opposition relation between the chambers
of the two buildings. This opposition relation shares many important
properties with the opposition relation on the chambers of a spherical building.
Thus, twin buildings appear to be natural generalisations of spherical
buildings with infinite Weyl groups.

One of the most celebrated results in the theory of abstract buildings
is Tits' classification of irreducible spherical buildings of rank at least $3$
in \cite{Ti74}. The decisive step in this classification is the proof of
a local-to-global result for spherical buildings. In his survey paper \cite{Ti92}
Tits proves several results that are inspired by his strategy in the
spherical case and he discusses several obstacles for obtaining
a similar local-to-global result for twin buildings. A first observation in this discussion
is that the local-to-global principle seems to be valid only for
$2$-spherical twin buildings. But even in this case the question about
the validity of the local-to-global principle remained open. Based on Tits'
contributions in \cite{Ti92} the local-to-global principle was proved in \cite{MR95}
for $2$-spherical twin buildings that satisfy an additional assumption,
called Condition $\co$. In \cite{MR95} Condition $\co$ is discussed in 
some detail and it turns out that it is rather
mild. On the other hand, it follows from that discussion that there
are affine twin buildings of type $\widetilde{C}_2$ that do not satisfy Condition $\co$.

The question whether the local-to-global principle for $2$-spherical buildings holds without 
Condition $\co$ is still open at present. The main result of this paper
is a contribution to the local-to-global principle without assuming any additional
condition.
It was proved independently by A.C. in [Ch00] and B.M. in [Mu97] but 
never published.
In the present article we follow the basic strategy of these references.
However, several contributions to the theory of twin buildings that 
have been made in the meantime
provided various improvements of the arguments and exposition.
Our motivation to publish
the paper at this point is provided by the fact that it can be used
to prove the local-to-global principle for $2$-spherical twin buildings
under a weaker assumption than Condition $\co$. This yields in
particular the local-to-global principle for all affine twin buildings of
rank at least $3$ and in particular for those which do not satisfy Condition $\co$.
This will be published in a subsequent paper. Thus,
the present paper should be seen as the first in a series of two papers
in which we intend to improve the main result of \cite{MR95}.

\subsection*{The main result}

In order to give the precise statement of the main result it is convenient to fix some notation.

Let $(W,S)$ be a Coxeter system. We call $(W,S)$ \textit{$2$-spherical} if
$st$ has finite order for all $s,t \in S$.

A \textit{building of type $(W,S)$} is a pair $\Delta = (\C,\delta)$ consisting
of a non-empty set $\C$ and a mapping $\delta: \C \times \C \longrightarrow W$ (see Section \ref{Sectionpreliminaries} for the precise definition).
The elements of $\C$ are called the \textit{chambers} of $\Delta$ and the mapping $\delta$
is called the \textit{Weyl-distance}. For $J \subseteq S$ and $c \in \C$, the set
$R_J(c) := \{ d \in \C \mid \delta(c,d) \in \langle J \rangle \}$ is called the \textit{$J$-residue}
of $c$ and for $s \in S$ the set $\P_s(c) := R_{\{s\}}(c)$ is called the \textit{$s$-panel} of $c$. The set
\[ E_2(c) := \bigcup\limits_{J \subseteq S, \vert J \vert \leq 2} R_J(c) \]
is called the \textit{foundation of $\Delta$ at $c$}.
The building $\Delta$ is said to be \textit{thick} if $\vert \P_s(c) \vert \geq 3$ for all $(s,c) \in S \times \C$.

A \textit{twin building of type $(W,S)$} is a triple $\Delta = (\Delta_+,\Delta_-,\delta_*)$
consisting of two buildings $\Delta_+ =(\C_+,\delta_+)$ and $\Delta_-=(\C_-,\delta_-)$ of type $(W,S)$
and a \textit{codistance function} (or \textit{twinning})
\[ \delta_*: (\C_+ \times \C_-) \cup (\C_- \times \C_+) \longrightarrow W \]
and we refer to Section \ref{Sectiontwinbuildings} for the precise definition.
For a chamber $c \in \C_+$ (resp. $c \in \C_-$) the set
$E_2(c)$ denotes its foundation of $\Delta_+$ (resp. $\Delta_-$) and
$\Delta$ is \textit{thick} if $\Delta_+$ and $\Delta_-$ are thick. Two chambers $c_+ \in \C_+$
and $c_- \in \C_-$ are said to be \textit{opposite in $\Delta$} if $\delta_*(c_+,c_-) = 1_W$.

Let $\Delta=((\C_+,\delta_+),(\C_-,\delta_-),\delta_*)$
and $\Delta'=((\C_+',\delta_+'),(\C_-',\delta_-'),\delta_*')$ be twin buildings of type $(W,S)$
and let ${\cal X} \subseteq \C_+ \cup \C_-, {\cal X}' \subseteq \C_+' \cup \C_-'$ be sets of
chambers of $\Delta$ and $\Delta'$. An \textit{isometry from ${\cal X}$ to ${\cal X}'$}
is a bijection from ${\cal X}$ onto ${\cal X}'$ which preserves signs and the Weyl-distance
(resp. codistance) for each pair $(x,y) \in {\cal X}^2$.

We are now in the position to give the precise statement of our main result.

\medskip
\noindent
\textbf{Main result:} Let $(W,S)$ be a $2$-spherical Coxeter system and let
$\Delta=((\C_+,\delta_+),(\C_-,\delta_-),\delta_*)$
and $\Delta'=((\C_+',\delta_+'),(\C_-',\delta_-'),\delta_*')$ be thick twin buildings of type $(W,S)$.
Let $c_+ \in \C_+, c_- \in \C_-$ be opposite chambers in $\Delta$ and
let $c_+' \in \C_+',c_-' \in \C_-'$ be opposite chambers in $\Delta'$.

Then each isometry $$\varphi: E_2(c_+) \cup \{ c_- \} \rightarrow E_2(c_+') \cup \{ c_-' \}$$
extends to an isometry
$$\psi: \C_+ \cup E_2(c_-) \rightarrow \C_+' \cup E_2(c_-').$$

\medskip
\noindent
Several remarks on the main result of this paper are in order.

\smallskip
\noindent
\textbf{$1$.} Note that our main result does not assert the uniqueness of the extension $\psi$. At present,
the uniqueness of $\psi$ is an open question that is most relevant for a possible proof of the
local-to-global principle. Indeed, the key observation in \cite{MR95} was that the extension $\psi$
is unique if $\Delta$ satisfies Condition $\co$.

\smallskip
\noindent
\textbf{$2$.} A slightly weaker version of our main result was stated by Tits in the early $1990$s
(as Th\'eor\`eme $1$ in \cite{Ti89/90} and as Theorem $2$ in \cite{Ti92}) and an outline
of a proof is given in both references. However, as pointed out in Paragraph $2.8$ 
of \cite{Ti97/98}, one of the claims made in the outline remains unclear. That our main result holds for twin buildings satisfying Condition $\co$ was verified by Ronan (see Theorem $(7.5)$ in \cite{Ro00}).

\smallskip
\noindent
\textbf{$3$.} The proof of the main result combines an idea of Tits given in the outline
mentioned in the previous remark with a technique that he used in \cite{Ti74}.
More concretely, for a chamber $c$ and an apartment containing $c$ in a twin building
one can define two retraction mappings. We call them $\pi$- and $\omega$-retractions.
The outline in \cite{Ti89/90} and \cite{Ti92} uses $\pi$-retractions and $\omega$-retractions
are used in \cite{Ti74} for the proof of the local-to-global principle for spherical buildings.
The key observation in this paper is that the main result can proved by using them both.

\section{Preliminaries}\label{Sectionpreliminaries}

\subsection*{Coxeter system}

Let $S$ be a set. A \textit{Coxeter matrix} over $S$ is a matrix $M = (m_{st})_{s, t \in S}$, whose entries are in $\NN \cup \{\infty\}$ such that $m_{ss} = 1$ for all $s\in S$ and $m_{st} = m_{ts} \geq 2$ for all $s \neq t\in S$. For $J \subseteq S$ we set $M_J := (m_{st})_{s, t \in J}$.  The \textit{Coxeter diagram} corresponding to $M$ is the labeled graph $(S, E(S))$, where $E(S) = \{ \{s, t \} \mid m_{st}>2 \}$ and where each edge $\{s,t\}$ is labeled by $m_{st}$ for all $s, t \in S$. As the Coxeter matrix and the corresponding Coxeter diagram carry the same information we do not distinguish between them formally. We call the Coxeter diagram \textit{irreducible}, if the underlying graph is connected, and we call it \textit{$2$-spherical}, if $m_{st} <\infty$ for all $s,t \in S$. The \textit{rank} of a Coxeter diagram is the cardinality of the set of its vertices.

Let $M = (m_{st})_{s, t \in S}$ be a Coxeter matrix over a set $S$. A \textit{Coxeter system of type $M$} is a pair $(W, S)$ consisting of a group $W$ and a set $S \subseteq W$ of generators of $W$ such that the set $S$ and the relations $\left(st \right) ^{m_{st}}$ for all $s,t\in S$ constitute a presentation of $W$.

Let $(W, S)$ be a Coxeter system of type $M$. The pair $(\langle J \rangle, J)$ is a Coxeter system of type $M_J$ (cf. \cite[Ch. IV, §$1$ Th\'eor\`eme $2$]{Bo68}). For an element $w\in W$ we put $\ell(w) := \min\{ k\in \NNN \mid \exists s_1, \ldots, s_k \in S: w = s_1 \cdots s_k \}$. The number $\ell(w)$ is called the \textit{length} of $w$. We call $J \subseteq S$ \textit{spherical} if $\langle J \rangle$ is finite. Given a spherical subset $J$ of $S$, there exists a unique element of maximal length in $\langle J \rangle$, which we denote by $r_J$ (cf. \cite[Corollary $2.19$]{AB08}); moreover, $r_J$ is an involution.

\begin{convention}
	For the rest of this paper let $S$ be a set, let $M$ be a Coxeter matrix over $S$ and let $(W, S)$ be a Coxeter system of type $M$.
\end{convention}

\subsection*{Buildings}

A \textit{building of type $(W, S)$} is a pair $\Delta = (\C, \delta)$ where $\C$ is a non-empty set and where $\delta: \C \times \C \to W$ is a \textit{distance function} satisfying the following axioms, where $x, y\in \C$ and $w = \delta(x, y)$:
\begin{enumerate}[label=(Bu\arabic*), leftmargin=*]
	\item $w = 1_W$ if and only if $x=y$;
	
	\item if $z\in \C$ satisfies $s := \delta(y, z) \in S$, then $\delta(x, z) \in \{w, ws\}$, and if, furthermore, $\ell(ws) = \ell(w) +1$, then $\delta(x, z) = ws$;
	
	\item if $s\in S$, there exists $z\in \C$ such that $\delta(y, z) = s$ and $\delta(x, z) = ws$.
\end{enumerate}
Let $\Delta = (\C, \delta)$ be a building of type $(W, S)$. The \textit{rank} of $\Delta$ is the rank of the underlying Coxeter system. The elements of $\C$ are called \textit{chambers}. Given $s\in S$ and $x, y \in \C$, then $x$ is called \textit{$s$-adjacent} to $y$, if $\delta(x, y) \in \langle s \rangle$. The chambers $x, y$ are called \textit{adjacent}, if they are $s$-adjacent for some $s\in S$. A \textit{gallery} joining $x$ and $y$ is a sequence $(x = x_0, \ldots, x_k = y)$ such that $x_{l-1}$ and $x_l$ are adjacent for any $1 \leq l \leq k$; the number $k$ is called the \textit{length} of the gallery.

Given a subset $J \subseteq S$ and $x\in \C$, the \textit{$J$-residue} of $x$ is the set $R_J(x) := \{y \in \C \mid \delta(x, y) \in \langle J \rangle \}$. Each $J$-residue is a building of type $(\langle J \rangle, J)$ with the distance function induced by $\delta$ (cf. \cite[Corollary $5.30$]{AB08}). A \textit{residue} is a subset $R$ of $\C$ such that there exists $J \subseteq S$ and $x\in \C$ with $R = R_J(x)$. Since the subset $J$ is uniquely determined by $R$, the set $J$ is called the \textit{type} of $R$ and the \textit{rank} of $R$ is defined to be the cardinality of $J$. A residue is called \textit{spherical} if its type is a spherical subset of $S$. A \textit{panel} is a residue of rank $1$. An \textit{$s$-panel} is a panel of type $\{s\}$ for $s\in S$. The building $\Delta$ is called \textit{thick}, if each panel of $\Delta$ contains at least three chambers.

Given $x\in \C$ and $k \in \NNN$ then $E_k(x)$ denotes the union of all residues of rank at most $k$ containing $x$. It is a fact, that the set $E_k(x)$ determines the chamber $x$ uniquely, if $k < \vert S \vert$.

Given $x\in \C$ and a $J$-residue $R \subseteq \C$, then there exists a unique chamber $z\in R$ such that $\ell(\delta(x, y)) = \ell(\delta(x, z)) + \ell(\delta(z, y))$ for any $y\in R$ (cf. \cite[Proposition $5.34$]{AB08}). The chamber $z$ is called the \textit{projection of $x$ onto $R$} and is denoted by $\proj_R x$. Moreover, if $z = \proj_R x$ we have $\delta(x, y) = \delta(x, z) \delta(z, y)$ for each $y\in R$.

A subset $\Sigma \subseteq \C$ is called \textit{convex} if $\proj_P c \in \Sigma$ for every $c\in \Sigma$ and every panel $P \subseteq \C$ which meets $\Sigma$. A subset $\Sigma \subseteq \C$ is called \textit{thin} if $P \cap \Sigma$ contains exactly two chambers for every panel $P \subseteq \C$ which meets $\Sigma$. An \textit{apartment} is a non-empty subset $\Sigma \subseteq \C$, which is convex and thin. It is a basic fact that in an apartment the map $\sigma_c: \Sigma \to W, x \mapsto \delta(c, x)$ is a bijection for any $c\in \C$.

\subsection*{Chamber systems}

Let $I$ be a set. A \textit{chamber system} over $I$ is a pair $\Cbf = \left( \C, (\sim_i)_{i\in I} \right)$ where $\C$ is a non-empty set whose elements are called \textit{chambers} and where $\sim_i$ is an equivalence relation on the set of chambers for each $i\in I$. Given $i\in I$ and $c, d \in \C$, then $c$ is called \textit{$i$-adjacent} to $d$ if $c \sim_i d$. The chambers $c, d$ are called \textit{adjacent} if they are $i$-adjacent for some $i \in I$.

A \textit{gallery} in $\Cbf$ is a sequence $(c_0, \ldots, c_k)$ such that $c_{\mu} \in \C$ for all $0 \leq \mu \leq k$ and such that $c_{\mu -1}$ is adjacent to $c_{\mu}$ for all $1 \leq \mu \leq k$. The number $k$ is called the \textit{length} of the gallery. Given a gallery $G = (c_0, \ldots, c_k)$, then we put $\beta(G) := c_0$ and $\epsilon(G) := c_k$. If $G$ is a gallery and if $c, d \in \C$ such that $c = \beta(G), d = \epsilon(G)$, then we say that $G$ is a \textit{gallery from $c$ to $d$} or \textit{$G$ joins $c$ and $d$}. The chamber system $\Cbf$ is said to be \textit{connected}, if for any two chambers there exists a gallery joining them. A gallery $G$ will be called \textit{closed} if $\beta(G) = \epsilon(G)$.

Given a gallery $G = (c_0, \ldots, c_k)$ then $G^{-1}$ denotes the gallery $(c_k, \ldots, c_0)$ and if $H = (c_0', \ldots, c_l')$ is a gallery such that $\epsilon(G) = \beta(H)$, then $GH$ denotes the gallery $(c_0, \ldots, c_k = c_0', \ldots, c_l')$.

Let $J$ be a subset of $I$. A \textit{$J$-gallery} is a gallery $(c_0, \ldots, c_k)$ such that for each $1 \leq \mu \leq k$ there exists an index $j \in J$ with $c_{\mu -1} \sim_j c_{\mu}$. Given two chambers $c, d$, then we say that $c$ is \textit{$J$-equivalent} with $d$ if there exists a $J$-gallery joining $c$ and $d$ and we write $c \sim_J d$ in this case. Given a chamber $c$ and a subset $J$ of $I$ then the set $R_J(c) := \{ d\in \C \mid c \sim_J d \}$ is called the \textit{$J$-residue} of $c$.

Let $\Delta = (\C, \delta)$ be a building of type $(W, S)$. Then we define the chamber system $\Cbf(\Delta)$ as follows: The set of chambers is identified with $\C$ and two chambers $x, y$ are defined to be $s$-adjacent if $\delta(x, y) \in \langle s \rangle$.

\subsection*{Homotopy of galleries and simple connectedness}

In the context of chamber systems there is the notation of $m$-homotopy and $m$-simple connectedness for each $m \in \NN$. In this paper we are only concerned with the case $m=2$. Therefore our definitions are always to be understood as a specialisation of the general theory to the case $m=2$.

Let $\Cbf = (\C, (\sim_i)_{i\in I})$ be a chamber system over a set $I$. Two galleries $G$ and $H$ are said to be \textit{elementary homotopic} if there exists two galleries $X, Y$ and two $J$-galleries $G_0, H_0$ for some $J \subseteq I$ of cardinality at most $2$ such that $G = XG_0Y, H = XH_0Y$. Two galleries $G, H$ are said to be \textit{homotopic} if there exists a finite sequence $G_0, \ldots, G_l$ of galleries such that $G_0 = G, G_l = H$ and such that $G_{\mu-1}$ is elementary homotopic to $G_{\mu}$ for all $1 \leq \mu \leq l$.

If two galleries $G, H$ are homotopic, then it follows by definition that $\beta(G) = \beta(H)$ and $\epsilon(G) = \epsilon(H)$. A closed gallery $G$ is said to be \textit{null-homotopic} if it is homotopic to the gallery $(\beta(G))$. The chamber system $(\C, (\sim_i)_{i\in I})$ is called \textit{simply connected} if it is connected and if each closed gallery is null-homotopic.

Let $\mathcal{X} \subseteq \C$ and let $\textbf{X} = \left(\mathcal{X}, (\sim_i)_{i\in I} \right)$ be the chamber system obtained by restricting the equivalence relations $\sim_i$ to $\mathcal{X}$. The subset $\mathcal{X}$ will be called \textit{simply connected} if the chamber system $\textbf{X}$ is simply connected.

\begin{proposition}\label{Mu97Proposition2.2}
	Let $\Delta$ be a building of type $(W, S)$. Then the chamber system $\Cbf(\Delta)$ is simply connected.
\end{proposition}
\begin{proof}
	This is \cite[$(4.3)$ Theorem]{Ro89}.
\end{proof}

\section{Twin buildings}\label{Sectiontwinbuildings}

\subsection*{Definitions and Notations}

Let $\Delta_+ = (\C_+, \delta_+), \Delta_- = (\C_-, \delta_-)$ be two buildings of the same type $(W, S)$. A \textit{codistance} (or a \textit{twinning}) between $\Delta_+$ and $\Delta_-$ is a mapping $\delta_* : \left( \C_+ \times \C_- \right) \cup \left( \C_- \times \C_+ \right) \to W$ satisfying the following axioms, where $\epsilon \in \{+,-\}, x\in \Cp, y\in \Cm$ and $w=\delta_*(x, y)$:
\begin{enumerate}[label=(Tw\arabic*), leftmargin=*]
	\item $\delta_*(y, x) = w^{-1}$;
	
	\item if $z\in \Cm$ is such that $s := \delta_{-\epsilon}(y, z) \in S$ and $\ell(ws) = \ell(w) -1$, then $\delta_*(x, z) = ws$;
	
	\item if $s\in S$, there exists $z\in \Cm$ such that $\delta_{-\epsilon}(y, z) = s$ and $\delta_*(x, z) = ws$.
\end{enumerate}
A \textit{twin building of type $(W, S)$} is a triple $\Delta = (\Delta_+, \Delta_-, \delta_*)$ where $\Delta_+, \Delta_-$ are buildings of type $(W, S)$ and where $\delta_*$ is a twinning between $\Delta_+$ and $\Delta_-$.

\begin{convention}
	For the rest of this paper let $\Delta = (\Delta_+, \Delta_-, \delta_*)$ be a twin building of type $(W, S)$ where $\Delta_+ = (\C_+, \delta_+)$ and $\Delta_- = (\C_-, \delta_-)$.
\end{convention}

We put $\C := \C_+ \cup \C_-$ and define the distance function $\delta: \C \times \C \to W$ by setting $\delta(x, y) := \delta_+(x, y)$ (resp. $\delta_-(x, y), \delta_*(x, y)$) if $x,y\in \C_+$ (resp. $x, y \in \C_-, (x, y) \in \Cp \times \Cm$ for some $\epsilon \in \{+,-\}$).

Given $x, y \in \C$ then we put $\ell(x, y) := \ell(\delta(x, y))$. If $\epsilon \in \{+,-\}$ and $x, y \in \Cp$, then we put $\ell_{\epsilon}(x, y) := \ell(\delta_{\epsilon}(x, y))$ and for $(x, y) \in \Cp \times \Cm$ we put $\ell_*(x, y) := \ell(\delta_*(x, y))$.

Let $\epsilon \in \{+,-\}$. For $x\in \Cp$ we put $x^{op} := \{ y\in \Cm \mid \delta_*(x, y) = 1_W \}$. It is a direct consequence of (Tw$1$) that $y\in x^{op}$ if and only if $x\in y^{op}$ for any pair $(x, y) \in \Cp \times \Cm$. If $y\in x^{op}$ then we say that $y$ is \textit{opposite} to $x$ or that \textit{$(x, y)$ is a pair of opposite chambers}.

Let $\overline{\C} := \{ (c_+, c_-) \in \C_+ \times \C_- \mid \delta_*(c_+, c_-) = 1_W \}$. Then $\left( \overline{\C}, (\sim_s)_{s\in S} \right)$ is a chamber system, where $(c_+, c_-) \in \overline{\C}$ is $s$-adjacent $(s\in S)$ to $(d_+, d_-) \in \overline{\C}$ if $c_{\epsilon}$ is $s$-adjacent to $ d_{\epsilon}$ in $\Cbf(\Delta_{\epsilon})$ for each $\epsilon \in \{+,-\}$. We denote this chamber system by $\Opp(\Delta)$. For $\overline{c} := (c_+, c_-) \in \overline{\C}$ we define $E_2(\overline{c}) := E_2(c_+) \cup E_2(c_-)$.

A \textit{residue} (resp. \textit{panel}) of $\Delta$ is a residue (resp. panel) of $\Delta_+$ or $\Delta_-$; given a residue $R\subseteq \C$ then we define its type and rank as before. Two residues $R,T \subseteq \C$ are called \textit{opposite} if they have the same type and if there exists a pair of opposite chambers $(x, y)$ such that $x\in R, y\in T$.

Let $\epsilon \in \{+,-\}$, let $J$ be a spherical subset of $S$ and let $R$ be a $J$-residue of $\Delta_{\epsilon}$. Given a chamber $x\in \Cm$ then there exists a unique chamber $z\in R$ such that $\ell_*(x, y) = \ell_*(x, z) - \ell_{\epsilon}(z, y)$ for any chamber $y\in R$ (cf. \cite[Lemma $5.149$]{AB08}). The chamber $z$ is called the \textit{projection of $x$ onto $R$}; it will be denoted by $\proj_R x$. Moreover, if $z = \proj_R x$ we have $\delta_*(x, y) = \delta_*(x, z)\delta_{\epsilon}(z, y)$ for each $y\in R$.

Let $\Sigma_+ \subseteq \C_+$ and $\Sigma_- \subseteq \C_-$ be apartments of $\Delta_+$ and $\Delta_-$, respectively. Then the set $\Sigma = \Sigma_+ \cup \Sigma_-$ is called \textit{twin apartment} if $\vert x^{op} \cap \Sigma \vert = 1$ for each $x\in \Sigma$. If $(x, y)$ is a pair of opposite chambers, then there exists a unique twin apartment containing $x$ and $y$. We will denote it by $A(x, y)$ and for $\epsilon \in \{+,-\}$ we put $A_{\epsilon}(x, y) := A(x, y) \cap \Cp$. It is a fact that $A(x, y) = \{ z\in \C \mid \delta(z, x) = \delta(z, y) \}$ (cf. Proposition $5.179$ in \cite{AB08}).

\begin{lemma}\label{Mu97Lemma3.2}
	Let $\Sigma \subseteq \C$ be a twin apartment, let $x \in \Sigma$ and let $R$ be a spherical residue of $\Delta$ which meets $\Sigma$. Then $\proj_R x \in \Sigma$.
\end{lemma}
\begin{proof}
	This is \cite[Lemma $5.173$ $(6)$]{AB08}.
\end{proof}

\subsection*{Pairs of opposite spherical residues}

Throughout this subsection we assume that $R \subseteq \C_+, T \subseteq \C_-$ are opposite residues and that the type $J$ of $R$ and $T$ is spherical.

\begin{lemma}\label{Mu97Lemma3.3}
	For each $x\in R$ there exists $y\in T$ such that $x$ and $y$ are opposite and we have $\delta_*(u, v) \in \langle J \rangle$ for all $(u, v) \in R \times T$.
\end{lemma}
\begin{proof}
	These are immediate consequences of \cite[Lemma $5.139$ $(1)$]{AB08}.
\end{proof}

\begin{lemma}\label{Mu97Lemma3.4}
	Let $(x, y) \in R \times T$. Then the following are equivalent:
	\begin{enumerate}[label=(\roman*), leftmargin=*]
		\item $\proj_T x = y$;
		\item $\delta_*(x, y) = r_J$;
		\item $\proj_R y = x$.
	\end{enumerate}
\end{lemma}
\begin{proof}
	Suppose $y= \proj_T x$ and let $z\in T$ be such that $\delta_-(y, z) = r_J$. Then $\ell_*(x, z) = \ell_*(x, y) - \ell(r_J)$ and hence $\ell_*(x, y) \geq \ell(r_J)$. As $\delta_*(x, y) \in \langle J \rangle$ by the previous Lemma, the claim follows.

	Suppose now that $\delta_*(x, y) = r_J$ and let $z := \proj_T x$. Since $\ell_*(x, z) \geq \ell_*(x, y) = \ell(r_J)$ and $\delta_*(x, z) \in \langle J \rangle$, it follows that $\delta_*(x, z) = r_J$. Now $\ell(r_J) = \ell_*(x, y) = \ell_*(x, z) - \ell_-(z, y) = \ell(r_J) - \ell_-(z, y)$ which implies $z=y$.

	We have shown that $(i)$ and $(ii)$ are equivalent; the equivalence of $(ii)$ and $(iii)$ follows by symmetry and we are done.
\end{proof}

\begin{lemma}\label{Mu97Lemma3.5}
	The mappings $\proj_R^T: T \to R, x \mapsto \proj_R x$ and $\proj_T^R: R \to T, x \mapsto \proj_T x$ are bijections inverse to each other.
\end{lemma}
\begin{proof}
	This is Proposition $(4.3)$ in \cite{Ro00}.
\end{proof}

\subsection*{A technical result}

In this paragraph we prove a technical result which will be needed in the proof of Theorem \ref{Maintheorem}.

\begin{lemma}
	Let $c \in \C_{-\epsilon}, x \in \C_{\epsilon}$ be two opposite chambers
	and let $(x = d_0, d_1, \ldots, d_k, d_{k+1} = d)$ be a gallery such that $\ell_*(c,d_i) = i$ for
	each $0 \leq i \leq k$ and $\ell_*(c,d) \leq k$. Then there exist chambers $x',z \in \C_{\epsilon}$
	such that $x' \in c^{op}$, $\delta_*(c,z) = \delta(x,z) = \delta(x',z)$ and $\ell(x',d) < k+1$.
\end{lemma}
\begin{proof}
	We put $w := \delta_{\epsilon}(x,d_k)$ and remark that our assumption implies $w = \delta_*(c,d_k)$.
	Furthermore we put $s:=\delta_{\epsilon}(d_k,d)$ and let $P$ denote the $s$-panel containing $d_k$ and $d$. By our 
	assumptions we have $\delta_{\epsilon}(x,d) \in \{ w, ws \}$. We have two cases:
	
	\smallskip
	\noindent
	$\ell(ws) = \ell(w)-1$: As $\delta_*(c,d_k) = w$ it follows that $d_k = \proj_P c$ and $\delta_*(c,d) =ws$.
	Let $x' \in \C_{\epsilon}$ be a chamber such that $\delta_{\epsilon}(x',d)=ws$. Then we have $x' \in c^{op}$ and $\delta_{\epsilon}(x',d_k) = w = \delta_{\epsilon}(x,d_k) = \delta_*(c,d_k)$ and $\ell_{\epsilon}(x',d) = k-1 < k+1$. Thus the assertion follows by setting $z:= d_k$.
		
	\smallskip
	\noindent
	$\ell(ws) = \ell(w)+1$: We put $z:= \proj_P c$. As $\ell_*(c,d) \leq k$ it follows that $z \neq d$ and $\delta_*(c,d) = w$.
	Let $x' \in \C_{\epsilon}$ be a chamber such that $\delta_{\epsilon}(x',d) =w$. Then $\delta_{\epsilon}(x',z) = ws = \delta_{\epsilon}(x,z) = \delta_*(c,z)$,
	and $x' \in c^{op}$ and the assertion follows.
\end{proof}

\begin{lemma}\label{Mu97Lemma6.10}
	Let $\epsilon \in \{+, -\}, c\in \Cm$ and let $x, y \in c^{op}$. Then there exist $k\in \NN$, a sequence $x=x_0, \ldots, x_k = y$ of chambers in $c^{op}$ and a sequence $z_1, \ldots, z_k$ of chambers in $\Cp$ such that $\delta_*(c, z_{\lambda}) = \delta_{\epsilon}(x_{\lambda -1}, z_{\lambda}) = \delta_{\epsilon}(x_{\lambda}, z_{\lambda})$ for each $1 \leq \lambda \leq k$.
\end{lemma}
\begin{proof}
	Let $(x=d_0, \ldots, d_m = y)$ be a minimal gallery joining $x$ and $y$. We will prove the assertion by induction on $m:= \ell_{\epsilon}(x,y)$. Setting $z:=x=y$ the assertion is trivial for $m = 0$
	and we may assume that $m > 0$.
	
	Let $k := \max \{ 0 \leq i \leq m \mid \ell_*(c,d_i) = i \}$ and put $d := d_{k+1}$. By the previous lemma
	there are chambers $x',z \in \C_{\epsilon}$ such that $x' \in c^{op}$, $\delta_{\epsilon}(x,z) = \delta_{\epsilon}(x',z) = \delta_*(c,z)$
	and $\ell_{\epsilon}(x',d) \leq k$. It follows $\ell_{\epsilon}(x',y) < m$ and we may apply induction to $x'$ and $y$ in order to obtain
	the desired sequences $x = x_0,x_1=x',\ldots,x_k=y$ and $z_1=z,z_2,\ldots,z_k$ of chambers.
\end{proof}

\section{Isometries}

Let $(W, S)$ be $2$-spherical and of rank at least $3$. Let $\Delta$ be thick and let $\Delta' = (\Delta_+', \Delta_-', \delta_*')$ be a thick twin building of type $(W, S)$. We define $\C', \Delta_+', \Delta_-', \delta', \ell'$ as in the case of $\Delta$.

\subsection*{Definition and basic facts about isometries}

Let $\mathcal{X} \subseteq \C, \mathcal{X}' \subseteq \C'$. A mapping $\phi: \mathcal{X} \to \mathcal{X}'$ is called \textit{isometry} if the following conditions are satisfied:
\begin{enumerate}[label=(Iso\arabic*), leftmargin=*]
	\item The mapping $\phi$ is bijective.
	
	\item For $\epsilon \in \{+,-\}$ we have $\phi(\mathcal{X} \cap \Cp) \subseteq \Cp'$.
	
	\item If $x, y \in \mathcal{X}$ then $\delta'(\phi(x), \phi(y)) = \delta(x, y)$.
\end{enumerate}
Given $\mathcal{X} \subseteq \C, \mathcal{X}' \subseteq \C'$, an isometry $\phi: \mathcal{X} \to \mathcal{X}'$ and $(y, y') \in \C \times \C'$, then the pair $(y, y')$ will be called \textit{$\phi$-admissible} if the mapping $y \mapsto y'$ extends $\phi$ to an isometry from $\mathcal{X} \cup \{y\}$ onto $\mathcal{X}' \cup \{y'\}$. In particular, $(x, \phi(x))$ is $\phi$-admissible for any $x\in \mathcal{X}$. For $x, y \in \mathcal{X}$ with $(x, y) \in \overline{\C}$ we define $\phi((x, y)) := (\phi(x), \phi(y))$. Since the building has rank at least three it is a fact that for $(x, x') \in \C \times \C'$ and $\phi: E_2(x) \to E_2(x')$ an isometry, we have $\phi(x) = x'$.

\begin{lemma}\label{ZusammengesetzteIsometrie}
	Let $\mathcal{S}, \mathcal{X} \subseteq \C, \mathcal{S}', \mathcal{X}' \subseteq \C'$ be such that $\mathcal{S} \cap \mathcal{X} = \emptyset$ and $\mathcal{S}' \cap \mathcal{X}' = \emptyset$. Let $\phi: \mathcal{S} \to \mathcal{S}'$ and $\psi: \mathcal{X} \to \mathcal{X}'$ be two isometries such that $(z, \psi(z))$ is $\phi$-admissible for any $z\in \mathcal{X}$. Then the mapping
	\[ \phi \cup \psi: \mathcal{S} \cup \mathcal{X} \to \mathcal{S}' \cup \mathcal{X}', x\mapsto \begin{cases}
	\phi(x) & \text{ if } x\in \mathcal{S}, \\
	\psi(x) & \text{ if } x\in \mathcal{X}.
	\end{cases} \]
	is an isometry.
\end{lemma}
\begin{proof}
	Let $\Phi := \phi \cup \psi$. Clearly, $\Phi$ is a bijection satisfying (Iso$2$). It suffices to show, that $\delta(x, y) = \delta'(\Phi (x), \Phi (y))$ for any $x \in \mathcal{S}, y\in \mathcal{X}$. Let $x\in \mathcal{S}$ and $y\in \mathcal{X}$. Then we have $\delta'(\Phi (x), \Phi (y)) = \delta'(\phi(x), \psi(y)) = \delta(x, y)$, because $(y, \psi(y))$ is $\phi$-admissible. This finishes the claim.
\end{proof}

\begin{lemma}\label{Mu97Lemma4.2}
	Let $J$ be a spherical subset of $S$, let $R \subseteq \C, R' \subseteq \C'$ be $J$-residues, let $\phi: R \to R'$ be an isometry, and let $(x, x')$ be a $\phi$-admissible pair. Then $\phi(\proj_R x) = \proj_{R'} x'$.
\end{lemma}
\begin{proof}
	This is Lemma $(4.4)$ of \cite{Ro00}.
\end{proof}

\begin{lemma}\label{Mu97Lemma4.3}
	Let $J$ be a spherical subset of $S$, let $R_+, R_- \subseteq \C$ (resp. $R_+', R_-' \subseteq \C'$) be opposite $J$-residues in $\Delta$ (resp. $\Delta'$), let $\phi: R_+ \cup R_- \to R_+' \cup R_-'$ be an isometry and let $\epsilon \in \{+,-\}$. Then $\phi(x) = \proj_{R_{\epsilon}'} \phi(\proj_{R_{-\epsilon}} x)$ for each $x\in R_{\epsilon}$.
\end{lemma}
\begin{proof}
	This is a consequence of the previous Lemma and Lemma \ref{Mu97Lemma3.5}.
\end{proof}

\begin{lemma}\label{Mu97Lemma4.4}
	Let $x\in \C, x' \in \C'$, let $\Sigma \subseteq \C$ be an apartment containing $x$ and let $\phi,\psi: E_2(x) \to E_2(x')$ be two isometries which agree on $E_1(x)$. If they also agree on $\Sigma \cap E_2(x)$, then we have $\phi = \psi$.
\end{lemma}
\begin{proof}
	For each subset $J$ of $S$ of cardinality $2$ we denote the restriction of $\phi$ (resp. $\psi$) on $R_J(x)$ by $\phi_J$ (resp. $\psi_J$).
	
	Let $J \subseteq S$ be of cardinality $2$ and let $\Sigma$ be as in the statement. Then $\phi_J$ and $\psi_J$ agree on $\Sigma \cap R_J(x)$ which is an apartment of $R_J(x)$. The claim follows from Theorem $4.1.1$ in \cite{Ti74}.
\end{proof}

\begin{lemma}\label{Mu97Lemma4.5}
	Let $\phi_+: \C_+ \to \C_+'$ be a map and let $\left( \phi_x: E_2(x) \to E_2(\phi_+(x)) \right)_{x \in \C_+}$ be a family of isometries such that $\phi_x$ and $\phi_y$ agree on $E_2(x) \cap E_2(y)$ whenever $x, y \in \C_+$ are adjacent. Then $\phi_+$ is an isometry and $\phi_x$ is the restriction of $\phi_+$ on $E_2(x)$ for each $x\in \C_+$.
\end{lemma}
\begin{proof}
	Let $x, y \in \C_+$ be such that $y \in E_2(x)$, then we can find a gallery $(x = x_0, \ldots, x_k = y)$ in a rank $2$ residue containing $x$ and $y$. It follows that $y \in E_2(x_{\lambda})$ for each $0 \leq \lambda \leq k$ and using induction one obtains $\phi_x(y) = \phi_y(y) = \phi_+(y)$. This shows that $\phi_x$ coincides with the restriction of $\phi_+$ on $E_2(x)$.
	
	Now we will show that $\phi_+$ is surjective. Let $y' \in \C_+'$. Let $x \in \C_+$ and let $x' := \phi_+(x)$. As $\phi_x: E_2(x) \to E_2(x')$ is an isometry, it follows that $E_2(x') \subseteq \phi_+(\C_+)$. By induction on the length of a minimal gallery joining $x'$ and $y'$ in $\C_+'$ it follows that $y' \in \phi_+(\C_+')$ and hence the surjectivity of $\phi_+$.
	
	The restriction of $\phi_+$ on the rank $2$ residues being isometries it follows that $\phi_+: \Cbf(\Delta_+) \to \Cbf(\Delta_+')$ is a $2$-covering. Now the injectivity of $\phi_+$ follows from Proposition \ref{Mu97Proposition2.2}.
\end{proof}

\begin{lemma}\label{Mu97Lemma4.6}
	Let $\phi_+: \C_+ \to \C_+'$ be an isometry, let $(x, x') \in \C_- \times \C_-'$ and suppose that $\phi_+(x^{op}) \subseteq (x')^{op}$. Then $(x, x')$ is a $\phi_+$-admissible pair.
\end{lemma}
\begin{proof}
	This is Lemma $(7.4)$ in \cite{Ro00}.
\end{proof}

\subsection*{Main results on local extensions of isometries}

In this subsection we let $\overline{c} := (c_+, c_-) \in \overline{\C}, \overline{c}' := (c_+', c_-') \in \overline{\C'}$.

\begin{proposition}\label{Mu97Proposition4.7}
	Let $\phi: E_2(c_+) \cup \{c_-\} \to E_2(c_+') \cup \{c_-'\}$ be an isometry. Then $\phi$ extends uniquely to an isometry from $E_2(c_+) \cup E_2(c_-)$ onto $E_2(c_+') \cup E_2(c_-')$.
\end{proposition}
\begin{proof}
	For a proof see Proposition $(6.2)$ of \cite{Ro00}.
\end{proof}

\begin{proposition}\label{Mu97Proposition4.8}
	Let $\overline{d} \in \overline{\C}$ such that $\overline{c}$ is adjacent to $\overline{d}$ in $\Opp(\Delta)$ and let $\phi: E_2(\overline{c}) \to E_2(\overline{c}')$ be an isometry. Then there exists a unique isometry $\psi: E_2(\overline{d}) \to E_2(\phi(\overline{d}))$ such that $\phi$ and $\psi$ agree on the intersection of their domains.
\end{proposition}
\begin{proof}
	This is Proposition $(6.4)$ of \cite{Ro00}.
\end{proof}

\begin{theorem}\label{Mu97Theorem4.9}
	Let $J$ be a subset of $S$ of cardinality at most $2$ and let $R_{\pm} := R_J(c_{\pm})$. Let $\overline{R} := \left( R_+ \times R_- \right) \cap \overline{\C}$ and let $\phi: E_2(\overline{c}) \to E_2(\overline{c}')$ be an isometry. Then there exists a unique system of isometries $\left( \phi_{\overline{x}}: E_2(\overline{x}) \to E_2(\phi(\overline{x})) \right)_{\overline{x} \in \overline{R}}$ such that the following is satisfied:
	\begin{enumerate}[label=(\alph*), leftmargin=*]
		\item $\phi_{\overline{c}} = \phi$;
		
		\item If $\overline{x}, \overline{y} \in \overline{R}$ are adjacent in $\Opp(\Delta)$, then $\phi_{\overline{x}}$ and $\phi_{\overline{y}}$ agree on the intersection of their domains.
	\end{enumerate}
\end{theorem}
\begin{proof}
	This is a consequence of Proposition $(6.6)$ and Corollary $(6.7)$ in \cite{Ro00}.
\end{proof}

\subsection*{Using $\Opp(\Delta)$ to extend isometries}

Let $\overline{c} \in \overline{\C}, \overline{c}' \in \overline{\C'}, \phi: E_2(\overline{c}) \to E_2(\overline{c}')$ be an isometry and let $\overline{G} = (\overline{c} = \overline{x}_0, \ldots, \overline{x}_k = \overline{d})$ be a gallery in $\Opp(\Delta)$. Then - by Proposition \ref{Mu97Proposition4.8} - we obtain recursively a unique chamber $\overline{d}'_{\phi, \overline{G}}$ and a unique isometry $\phi_{\overline{d}, \overline{G}}: E_2(\overline{d}) \to E_2(\overline{d}'_{\phi, \overline{G}})$.

\begin{lemma}\label{Mu97Lemma5.2}
	The following hold:
	\begin{enumerate}[label=(\roman*), leftmargin=*]
		\item Given any gallery $\overline{G}$ starting at $\overline{c}$, then $\overline{c}'_{\phi, \overline{G} \, \overline{G}^{-1}} = \overline{c}'$ and $\phi_{\overline{c}, \overline{G} \, \overline{G}^{-1}} = \phi$.
		
		\item Given any closed gallery $\overline{G}$ in a rank $2$ residue of $\overline{c}$, then $\overline{c}'_{\phi, \overline{G}} = \overline{c}'$ and $\phi_{\overline{c}, \overline{G}} = \phi$.
		
		\item If two galleries $\overline{G}, \overline{H}$ joining $\overline{c}$ and $\overline{d}$ are homotopic, then $\overline{d}'_{\phi, \overline{G}} = \overline{d}'_{\phi, \overline{H}}$ and $\phi_{\overline{d}, \overline{G}} = \phi_{\overline{d}, \overline{H}}$.
	\end{enumerate}
\end{lemma}
\begin{proof}
	Part $(i)$ follows from the uniqueness assertion in Proposition \ref{Mu97Proposition4.8}; part $(ii)$ follows from Theorem \ref{Mu97Theorem4.9}, and part $(iii)$ is a consequence of $(i)$ and $(ii)$.
\end{proof}

\begin{proposition}\label{Mu97Proposition5.3}
	Let $\overline{\mathcal{X}} \subseteq \overline{\C}$ be simply connected and suppose that $\overline{c} \in \overline{\mathcal{X}}$. Then there exists a mapping $\overline{\phi}: \overline{\mathcal{X}} \to \overline{\C'}$ and a system of isometries $\left( \phi_{\overline{x}}: E_2(\overline{x}) \to E_2(\overline{\phi}(\overline{x})) \right)_{\overline{x} \in \overline{\mathcal{X}}}$ such that $\phi_{\overline{c}} = \phi$ and such that $\phi_{\overline{x}}$ and $\phi_{\overline{y}}$ agree on the intersection of their domains for any two adjacent chambers $\overline{x}, \overline{y} \in \overline{\mathcal{X}}$. The mapping $\overline{\phi}$ and the family of  isometries $\phi_{\overline{x}}$ are uniquely determined by these properties.
\end{proposition}
\begin{proof}
	As $\overline{\mathcal{X}}$ is simply connected it is connected by definition. Given $\overline{x} \in \overline{\mathcal{X}}$ we obtain for each gallery $\overline{G}$ joining $\overline{c}$ and $\overline{x}$ a unique chamber $\overline{x}'_{\phi, \overline{G}}$ and an isometry $\phi_{\overline{x}, \overline{G}}: E_2(\overline{x}) \to E_2(\overline{x}'_{\phi, \overline{G}})$. It follows by part $(iii)$ of the previous Lemma that $\overline{x}'_{\overline{G}} = \overline{x}'_{\overline{H}}$ for any two galleries $\overline{G}, \overline{H}$ from $\overline{c}$ to $\overline{x}$ because $\overline{\mathcal{X}}$ is simply connected. Thus we obtain a mapping $\overline{\phi}$ and a system of isometries $\left( \phi_{\overline{x}} \right)_{\overline{x} \in \overline{\mathcal{X}}}$.
		
	Let $\overline{x}, \overline{y} \in \overline{\mathcal{X}}$ be adjacent. By considering a gallery joining $\overline{c}$ and $\overline{x}$ which passes through $\overline{y}$ it is seen that it follows by construction that $\phi_{\overline{x}}$ and $\phi_{\overline{y}}$ agree on the intersection of their domains.
	
	The uniqueness of $\overline{\phi}$ and $\left( \phi_{\overline{x}} \right)_{\overline{x} \in \overline{\mathcal{X}}}$ follows from the uniqueness assertion of Proposition \ref{Mu97Proposition4.8} and an obvious induction.
\end{proof}

\section{Retractions}

\begin{convention}
	For the rest of this paper let $(W, S)$ be $2$-spherical and of rank at least $3$. Furthermore, let $\Delta$ be thick and let $\Delta' = (\Delta_+', \Delta_-', \delta_*')$ be a thick twin building of type $(W, S)$. We define $\C', \Delta_+', \Delta_-', \delta', \ell'$ as in the case of $\Delta$.
\end{convention}

\subsection*{$\pi$-retractions}

Let $c\in \C_-$, let $\Sigma \subseteq \C_-$ be an apartment of $\Delta_-$ containing $c$ and put $\gamma := (c, \Sigma)$. Then we define the mapping $\pi_{\gamma}: \C_+ \to \Sigma$ via $\delta_-(c, \pi_{\gamma}(x)) = \delta_*(c, x)$ for all $x\in \C_+$ and we put $\Pi_{\gamma} := \{ (x, \pi_{\gamma}(x)) \mid x\in \C_+ \}$.

\begin{lemma}\label{pi123}
	Let $\gamma = (c, \Sigma)$ be as above, then the following hold:
	\begin{enumerate}[label=($\pi$\arabic*), leftmargin=*]
		\item $\pi_{\gamma}$ preserves $s$-adjacency.
		
		\item The chamber $\pi_{\gamma}(x)$ is opposite to $x$ for each chamber $x\in \C_+$.
		
		\item Given $x\in \C_+$, then $c \in A(x, \pi_{\gamma}(x))$.
	\end{enumerate}
\end{lemma}
\begin{proof}
	The first two assertions are proved in Lemma $(7.1)$ in \cite{Ro00}. For the third assertion we notice that $\delta_-(c, \pi_{\gamma}(x)) = \delta_*(c, x)$ by definition and hence the claim follows.
\end{proof}

\begin{lemma}\label{Mu97Lemma3.6}
	Let $\gamma = (c, \Sigma)$ be as above, then the mapping $\C_+ \to \Pi_{\gamma}, x \mapsto (x, \pi_{\gamma}(x))$ is an $s$-adjacence preserving bijection. In particular, $\Pi_{\gamma}$ is a simply connected subset of $\overline{\C}$.
\end{lemma}
\begin{proof}
	The first statement is immediate from Lemma \ref{pi123}. The second follows from Proposition \ref{Mu97Proposition2.2}.
\end{proof}

\subsection*{$\omega$-retractions}

Let $\overline{c} := (c_+, c_-)$ be a pair of opposite chambers and let $\Sigma = A_-(c_+, c_-)$. Then we define the mapping $\omega_{\overline{c}}: \C_+ \to \Sigma$ via $\delta_-(c_-, \omega_{\overline{c}}(x)) = \delta_+(c_+, x)$ for all $x\in \C_+$. Furthermore we set $\Omega_{\overline{c}} := \{ (x, \omega_{\overline{c}}(x)) \mid x \in \C_+ \}$. A gallery $(\overline{x} = \overline{x}_0, \ldots, \overline{x}_k = \overline{y})$ in $\Opp(\Delta)$ will be called \textit{$\omega$-gallery} if there exists a chamber $\overline{c} \in \overline{\C}$ such that $\overline{x}_{\lambda} \in \Omega_{\overline{c}}$ for each $0 \leq \lambda \leq k$.

\begin{lemma}\label{omega123}
	Let $\overline{c} \in \overline{\C}$. Then the following hold:
	\begin{enumerate}[label=($\omega$\arabic*), leftmargin=*]
		\item $\omega_{\overline{c}}$ preserves $s$-adjacency.
		
		\item The chamber $\omega_{\overline{c}}(x)$ is opposite to $x$ for each chamber $x\in \C_+$.
		
		\item Given $x\in \C_+$, then $c_+ \in A(x, \omega_{\overline{c}}(x))$.
	\end{enumerate}
\end{lemma}
\begin{proof}
	Let $x, y \in \C_+$ and $s\in S$ such that $x, y$ are $s$-adjacent, and let $w\in W$ such that $\delta_+(c_+, x) = w$. Then $\delta_+(c_+, y) \in \{ w, ws \}$ by (Bu$2$). If $\delta_+(c_+, y) = w$ then $\delta_-(c_-, \omega_{\overline{c}}(x)) = \delta_-(c_-, \omega_{\overline{c}}(y))$. Since $c_-, \omega_{\overline{c}}(x), \omega_{\overline{c}}(y) \in \Sigma$ we obtain $\omega_{\overline{c}}(x) = \omega_{\overline{c}}(y)$. Now we assume that $\delta_+(c_+, y) = ws$. Let $P$ be the $s$-panel containing $\omega_{\overline{c}}(y)$. Since $\omega_{\overline{c}}(y) \in P \cap \Sigma$ we obtain $\vert P \cap \Sigma \vert = 2$ because any apartment is thin. Let $\omega_{\overline{c}}(y) \neq y' \in P \cap \Sigma$. Using (Bu$2$) we obtain $\delta_-(c_-, y') \in \{ w, ws \}$. Since $c_-, y', \omega_{\overline{c}}(x) \in \Sigma$ and $\delta_-(c_-, y') = \delta_-(c_-, \omega_{\overline{c}}(x))$, we obtain $y' = \omega_{\overline{c}}(x)$ as above. Thus $\omega_{\overline{c}}$ preserves $s$-adjacency.
	
	Let $x\in \C_+$ and $w\in W$ such that $\delta_+(x, c_+) = w$. Then $\delta_-(c_-, \omega_{\overline{c}}(x)) = \delta_+(c_+, x) = w^{-1}$. Since $\omega_{\overline{c}}(x) \in A(c_+, c_-)$ we have $\delta_*(\omega_{\overline{c}}(x), c_+) = \delta_-(\omega_{\overline{c}}(x), c_-) = w$. Now we have $\delta_*(\omega_{\overline{c}}(x), x) = ww^{-1} = 1_W$ by Lemma $5.140$ of \cite{AB08} and the claim follows.
	
	Let $x\in \C_+$. Since $\omega_{\overline{c}} \in A(c_+, c_-)$ we obtain $\delta_*(c_+, \omega_{\overline{c}}(x)) = \delta_-(c_-, \omega_{\overline{c}}(x))$. Furthermore, we have $\delta_+(c_+, x) = \delta_-(c_-, \omega_{\overline{c}}(x))$. Combining these two facts we obtain $c_+ \in A(x, \omega_{\overline{c}}(x))$ as required.
\end{proof}

\begin{lemma}\label{Mu97Lemma6.2}
	Let $P$ be an $s$-panel in $\Delta_+$, let $x, y \in P$ be such that $\ell_+(c_+, y) = \ell_+(c_+, x) +1$ and let $Q$ denote the $s$-panel of $\Delta_-$ containing $\omega_{\overline{c}}(x)$ and $\omega_{\overline{c}}(y)$. Then the following hold:
	\begin{enumerate}[label=(\roman*), leftmargin=*]
		\item $\proj_P c_+ = x$;
		
		\item $\proj_Q c_+ = \omega_{\overline{c}}(y)$;
		
		\item $\proj_P \omega_{\overline{c}}(y) = x$;
		
		\item $\proj_Q x = \omega_{\overline{c}}(y)$.
	\end{enumerate}
\end{lemma}
\begin{proof}
	Part $(i)$ follows from $\ell_+(c_+, y) = \ell_+(c_+, x) +1$. Since $c_+ \in A(y, \omega_{\overline{c}}(y)) \cap A(x,  \omega_{\overline{c}}(x))$ we have $\ell_*(c_+, \omega_{\overline{c}}(y)) = \ell_+(c_+, y) = \ell_+(c_+, x) +1 = \ell_*(c_+, \omega_{\overline{c}}(x)) +1$ which yields part $(ii)$. To prove part $(iii)$ we use the fact that $c_+ \in A(y, \omega_{\overline{c}}(y))$. As $\proj_P c_+ = x$, it follows by Lemma \ref{Mu97Lemma3.2}, that $x \in A(y, \omega_{\overline{c}}(y))$. Applying Lemma \ref{Mu97Lemma3.2} again we obtain that $\proj_P \omega_{\overline{c}}(y) \in \{x, y\}$, since $A_+(y, \omega_{\overline{c}}(y))$ is thin. As $\ell_*(y, \omega_{\overline{c}}(y)) =0$ we have $\proj_P \omega_{\overline{c}}(y) = x$ as claimed. Part $(iv)$ follows now from part $(iii)$ and Lemma \ref{Mu97Lemma3.5}.
\end{proof}

\begin{lemma}\label{Mu97Lemma6.3}
	The mapping $\C_+ \to \Omega_{\overline{c}}: x \mapsto (x, \omega_{\overline{c}}(x))$ is an $s$-adjacence preserving bijection between $\C_+$ and $\Omega_{\overline{c}}$. In particular, $\Omega_{\overline{c}}$ is a simply connected subset  of $\overline{\C}$.
\end{lemma}
\begin{proof}
	The first statement is immediate from Lemma \ref{omega123}. The second follows from Proposition \ref{Mu97Proposition2.2}.
\end{proof}

\begin{lemma}\label{Mu97Lemma6.5}
	Let $c \in \C_-$, let $\Sigma$ be an apartment of $\Delta_-$ containing $c$ and let $\gamma := (c, \Sigma)$. Let $x, y \in c^{op}$ and suppose that there exists a chamber $z\in \C_+$ such that $\delta_+(x, z) = \delta_*(c, z) = \delta_+(y, z)$. Then there exists an $\omega$-gallery joining $(x, c)$ and $(y, c)$ in $\Pi_{\gamma} \cap \Omega_{(z, \pi_{\gamma}(z))}$.
\end{lemma}
\begin{proof}
	We put $\overline{z} := (z, \pi_{\gamma}(z))$. Then we obtain that $\omega_{\overline{z}}(z) = \pi_{\gamma}(z), \omega_{\overline{z}}(x) = \pi_{\gamma}(x) = c$ and $\delta_-(\omega_{\overline{z}}(x), \omega_{\overline{z}}(z)) = \delta_+(x, z) = \delta_-(\pi_{\gamma}(x), \pi_{\gamma}(z))$. Since $\pi_{\gamma}$ and $\omega_{\overline{z}}$ preserve $s$-adjacency by ($\pi 1$) and ($\omega 1$), it follows that they map any chamber on a minimal gallery joining $x$ any $z$ to a chamber on a minimal gallery joining $\pi_{\gamma}(x) = \omega_{\overline{z}}(x)$ to $\pi_{\gamma}(z) = \omega_{\overline{z}}(z)$. Thus we obtain $\pi_{\gamma}(v) = \omega_{\overline{z}}(v)$ for each chamber $v$ on a minimal gallery joining $x$ and $z$. The same is true for $y$ instead of $x$ and we obtain $\pi_{\gamma}(u) = \omega_{\overline{z}}(u)$ for each chamber $u$ on a minimal gallery joining $y$ and $z$. This yields the claim.
\end{proof}

\section{Constructing an isometry}

We recall that the set $S$ has at least three elements. In this subsection let $\overline{c} := (c_+, c_-) \in \overline{\C}, \overline{c}' := (c_+', c_-') \in \overline{\C'}$ and let $\phi: E_2(\overline{c}) \to E_2(\overline{c}')$ be an isometry. We set $\Sigma := A_-(c_+, c_-), \Sigma' := A_-(c_+', c_-')$ and denote the unique isometry from $\Sigma$ onto $\Sigma'$ extending the mapping $c_- \mapsto c_-'$ by $\alpha$. We set $\omega := \omega_{\overline{c}}, \omega' := \omega_{\overline{c}'}$ and $\Omega := \Omega_{\overline{c}}$. For $x\in \C_+$ we put $\overline{x} := (x, \omega(x))$.
	
By Lemma \ref{Mu97Lemma6.3} and Proposition \ref{Mu97Proposition5.3} we get a mapping $\overline{\phi}: \Omega \to \overline{\C'}$ and a system of isometries $\left( \phi_{\overline{x}}: E_2(\overline{x}) \to E_2(\overline{\phi}(\overline{x})) \right)_{x \in \C_+}$ such that
\begin{enumerate}[label=(\roman*), leftmargin=*]
	\item $\phi_{\overline{c}} = \phi$;
	
	\item $\phi_{\overline{x}}$ and $\phi_{\overline{y}}$ coincide on the intersection of their domains whenever $x, y$ are adjacent.
\end{enumerate}
Furthermore, we define the mapping $\phi_+: \C_+ \to \C_+', x \mapsto \phi_{\overline{x}}(x)$ and denote the restriction of $\phi_{\overline{x}}$ on $E_2(x)$ by $\phi_x$.

\begin{lemma}\label{Mu97Lemma6.7}
	The mapping $\phi_+$ is an isometry from $\C_+$ to $\C_+'$ and $\phi_x$ is the restriction of $\phi_+$ on $E_2(x)$ for each $x\in \C_+$.
\end{lemma}
\begin{proof}
	Given two adjacent chambers $x, y \in \C_+$, then $\overline{x}$ and $\overline{y}$ are adjacent. By property $(ii)$ above it follows that $\phi_x$ and $\phi_y$ coincide on $E_2(x) \cap E_2(y)$. This shows that $\phi_+$ and $\left( \phi_x \right)_{x\in \C_+}$ satisfy the conditions of Lemma \ref{Mu97Lemma4.5} and we are done.
\end{proof}

\begin{lemma}\label{Mu97Lemma6.8}
	Let $P \subseteq \C_+, P' \subseteq \C_+'$ be panels of $\Delta_+$ and $\Delta_+'$ having the same type and let $x := \proj_P c_+, x' := \proj_{P'} c_+'$. Suppose that $\delta_+(c_+, x) = \delta_+'(c_+', x')$ and let $\psi: E_2((x, \omega(x))) \to E_2((x', \omega'(x')))$ be an isometry. Given $y\in P$, then $\psi(\omega(y)) = \omega'(\psi(y))$.
\end{lemma}
\begin{proof}
	If $x=y$ there is nothing to prove, so we may assume that $x \neq y$. As $x = \proj_P c_+$ we have $\ell_+(c_+, y) = \ell_+(c_+, x) +1$. We put $y' := \psi(y)$. We have $y' \in P'$ and $x' \neq y'$ because $\psi$ is an isometry. As $x' = \proj_{P'} c_+'$ it follows $\ell_+'(c_+', y') = \ell_+'(c_+', x') +1$. Let $Q$ (resp. $Q'$) denote the panel containing $\omega(x)$ (resp. $\omega'(x')$) opposite to $P$ (resp. $P'$). By Lemma \ref{Mu97Lemma6.2} and Lemma \ref{Mu97Lemma4.3} it follows that $\psi(\omega(y)) = \proj_{Q'} \psi(\proj_P \omega(y)) = \proj_{Q'} \psi(x) = \proj_{Q'} x' = \omega'(y')$ which yields the claim.
\end{proof}

\begin{proposition}\label{Mu97Proposition6.6}
	Let $x, y \in \C_+$ be such that $\omega(x) = \omega(y)$, then the restrictions of $\phi_{\overline{x}}$ and $\phi_{\overline{y}}$ on $E_2(\omega(x))$ coincide.
\end{proposition}
\begin{proof}
	Using Lemma \ref{Mu97Lemma6.8} it follows by induction on $\ell_+(c_+, u)$ that $\phi_{\overline{u}}(\omega(u)) = \omega'(\phi_+(u))$ for each $u\in \C_+$. As $\phi_+$ is an isometry mapping $c_+$ onto $c_+'$ (cf. Lemma \ref{Mu97Lemma6.7}) it follows that $\alpha(\omega(u)) = \omega'(\phi_+(u))$ for each $u\in \C_+$.
	
	Let $u\in \C_+$ and let $z\in E_2(\omega(u)) \cap \Sigma$. Then there exists $v\in E_2(u)$ with $z=\omega(v)$. Let $(u=x_0, \ldots, x_k = v)$ be a gallery in a rank $2$ residue joining $u$ and $v$. It follows that $v\in E_2(x_{\lambda})$ and hence $z\in E_2(\omega(x_{\lambda}))$ for each $0 \leq \lambda \leq k$. Using property $(ii)$ of the system $\left( \phi_{\overline{u}} \right)_{u\in \C_+}$, it follows by induction on $k$ that $\phi_{\overline{u}}(z) = \phi_{\overline{v}}(z)$. Combining this with the previous considerations we obtain $\phi_{\overline{u}}(z) = \alpha(z)$ for each $z \in E_2(\omega(u)) \cap \Sigma$.
		
	We complete the proof of the proposition by induction on $\ell_+(c_+, x) = \ell_+(c_+, y)$. If $\ell_+(c_+, x) = 0$ then $x= c_+= y$ and there is nothing to prove. Let $\ell_+(c_+, x) >0$, then there exists $s\in S$ such that $\ell(\delta_+(c_+, x)s) = \ell_+(c_+, x) -1$. Let $P_x, P_y$ denote the $s$-panels containing $x$ and $y$, respectively, and put $x_1 := \proj_{P_x} c_+, y_1 := \proj_{P_y} c_+$. Then $\ell_+(c_+, x_1) = \ell_+(c_+, x) -1$ and we obtain $\omega(x_1) = \omega(y_1)$. Using property $(ii)$ of the system $\left( \phi_{\overline{x}} \right)_{x\in \C_+}$ and the induction assumption we obtain $\phi_{\overline{x}}(z) = \phi_{\overline{x}_1}(z) = \phi_{\overline{y}_1}(z) = \phi_{\overline{y}}(z)$ for each $z\in E_1(\omega(x))$. By the previous considerations we have that $\phi_{\overline{x}}$ and $\phi_{\overline{y}}$ agree on $E_2(\omega(x)) \cap \Sigma$ and therefore the claim follows from Lemma \ref{Mu97Lemma4.4}.
\end{proof}

A consequence of Proposition \ref{Mu97Proposition6.6} is the following corollary which will be needed in the next subsection.

\begin{corollary}\label{Mu97Corollary6.9}
	Let $\overline{x} := (x, z), \overline{y} := (y, z) \in \overline{\C}$, let $\overline{x}' \in \overline{\C'}$ and let $\psi: E_2(\overline{x}) \to E_2(\overline{x}')$ be an isometry. Let $\overline{G}$ be an $\omega$-gallery joining $\overline{x}$ and $\overline{y}$ in $\Opp(\Delta)$. Then $\psi_{\overline{y}, \overline{G}}$ and $\psi$ coincide on $E_2(z)$.
\end{corollary}
\begin{proof}
	Since $\omega(x) = z = \omega(y)$, the claim follows from Proposition \ref{Mu97Proposition6.6} and the definition of $\psi_{\overline{y}, \overline{G}}$.
\end{proof}

\subsection*{Proof of the main theorem}

\begin{theorem}\label{Maintheorem}
	Let $\overline{c} := (c_+, c_-) \in \overline{\C}, \overline{c}' := (c_+', c_-') \in \overline{\C'}$. Then every isometry $\phi: E_2(c_+) \cup \{c_-\} \to E_2(c_+') \cup \{c_-'\}$ extends to an isometry from $\C_+ \cup E_2(c_-)$ onto $\C_+' \cup E_2(c_-')$.
\end{theorem}
\begin{proof}
	By Proposition \ref{Mu97Proposition4.7} the isometry $\phi$ extends to an isometry from $E_2(\overline{c})$ onto $E_2(\overline{c}')$. We choose an apartment $\Sigma \subseteq \C_-$ containing $c_-$ and set $\pi := \pi_{(c_-, \Sigma)}, \Pi := \Pi_{(c_-, \Sigma)}$. For $x\in \C_+$ we put $\overline{x} := (x, \pi(x))$. By Lemma \ref{Mu97Lemma3.6} and  Proposition \ref{Mu97Proposition5.3} we obtain a mapping $\overline{\phi}: \Pi \to \overline{\C'}$ and a system of isometries $\left( \phi_{\overline{x}}: E_2(\overline{x}) \to E_2(\overline{\phi}(\overline{x})) \right)_{x\in \C_+}$ with the properties $(i)$ and $(ii)$ of the previous subsection. We define the mapping $\phi_+: \C_+ \to \C_+', x \mapsto \phi_{\overline{x}}(x)$ and we denote the restriction of $\phi_{\overline{x}}$ on $E_2(x)$ by $\phi_x$. Then $\phi_+$ is an isometry from $\C_+$ onto $\C_+'$ and $\phi_+, \phi_x$ agree on $E_2(x)$ for each $x\in \C_+$ by Lemma \ref{Mu97Lemma6.7}.
	
	Let $x, y \in c_-^{op}$. Using Lemma \ref{Mu97Lemma6.10} there exist $k\in \NN$, a sequence $x_0 := x, \ldots, x_k := y$ of chambers in $c_-^{op}$ and a sequence $z_1, \ldots, z_k$ of chambers in $\C_+$ such that $\delta_*(c_-, z_{\lambda}) = \delta_+(x_{\lambda -1}, z_{\lambda}) = \delta_+(x_{\lambda}, z_{\lambda})$ for each $1 \leq \lambda \leq k$. By Lemma \ref{Mu97Lemma6.5} there exists for any $1 \leq \lambda \leq k$ an $\omega$-gallery joining $(x_{\lambda -1}, c_-)$ and $(x_{\lambda}, c_-)$ in $\Pi_{\gamma} \cap \Omega_{(z_{\lambda}, \pi_{\gamma}(z_{\lambda}))}$. Now we obtain	that $\phi_{\overline{x}}, \phi_{\overline{y}}$ agree on $E_2(c_-)$ by Corollary \ref{Mu97Corollary6.9}. We let $\phi_-: E_2(c_-) \to E_2(c_-'), z \mapsto \phi_{\overline{x}}(z)$ denote this common restriction for some $x \in c_-^{op}$ and for $z\in E_2(c_-)$ we put $z' := \phi_-(z)$.
	
	We want to show now, that $\phi_+(z^{op}) \subseteq \left( z' \right)^{op}$ for each $z\in E_2(c_-)$. Let $v\in z^{op}$ then there exists $x\in c_-^{op}$ such that $v\in E_2(x)$ by Lemma \ref{Mu97Lemma3.3}. Since $\phi_+(v) = \phi_x(v), \phi_x = \phi_{\overline{x}} \vert_{E_2(x)}$ and since $\phi_{\overline{x}}$ is an isometry from $E_2(\overline{x})$ onto $E_2(\overline{\phi}(\overline{x}))$ whose restriction on $E_2(c_-)$ is $\phi_-$ it follows that $\phi_+(v) \in \left( z' \right)^{op}$.
	
	By Lemma \ref{Mu97Lemma4.6} the pair $(z, z') = (z, \phi_-(z))$ is $\phi_+$-admissible. Applying Lemma \ref{ZusammengesetzteIsometrie} to the isometries $\phi_+$ and $\phi_-$ we obtain an isometry $\phi_+ \cup \phi_-$ as required.
\end{proof}


\begin{thebibliography}{99}
	
	\bibitem[AB$08$]{AB08} P. Abramenko, K. S. Brown: \textit{Buildings - Theory and Applications}. Springer GTM $\textbf{248}$, $2008$.

	\bibitem[Bo$68$]{Bo68} N. Bourbaki: \textit{Groupes et Alg\'ebres de Lie, Chapitres $4,5$ et $6$}. Hermann, Paris, $1968$.
	
	\bibitem[Ch$00$]{Ch00} A. Chosson: \textit{Isom\'etries dans les immeubles jumel\'es et construction d'automorphismes de groupes de Kac-Moody}. Th\`ese de doctorat, Amiens, $2000$.

	
	\bibitem[Mu$97$]{Mu97} B. M\"uhlherr: \textit{An alternative proof of the extension theorem for twin buildings}. Preprint, Dortmund, $1997$, $17$ pages.

	\bibitem[MR$95$]{MR95} B. M\"uhlherr, M. Ronan: \textit{Local to Global Structure in Twin Buildings}. Invent. math. $\textbf{122}$, $1995$, $71-81$.

	\bibitem[Ro$89$]{Ro89} M. Ronan: \textit{Lectures on Buildings (Updated and Revised)}. University of Chicago Press, $2009$.

	\bibitem[Ro$00$]{Ro00} M. Ronan: \textit{Local isometries of twin buildings}. Math. Z. $\textbf{234}$, $2000$, $435-455$.

	\bibitem[Ti$74$]{Ti74} J. Tits: \textit{Buildings of spherical type and finite BN-pairs}. Lecture Notes in Mathematics $386$, Springer Verlag $1974$. $\nth{2}$ edition $1986$.


	\bibitem[Ti$89/90$]{Ti89/90} J. Tits: \textit{Immeubles jumel\'es (Suite)}. Annuaire du Coll\`ege de France, $90$e, ann\'ee, $1989-1990$, $87-103$.

	\bibitem[Ti$92$]{Ti92} J. Tits: \textit{Twin Buildings and groups of Kac-Moody type}. In M. W. Liebeck and J. Saxl (eds.), \textit{Groups, Combinatorics and Geometry} (Durham $1990$), London Math. Soc. Lecture Notes Ser. $\textbf{165}$, Cambridge University Press, $1992$, $249-286$.
	
	\bibitem[Ti$97/98$]{Ti97/98} J. Tits: \textit{Immeubles jumel\'es: th\'eor\`emes d'existence}. Annuaire du Coll\`ege de France, $1997-1998$, $97-112$.


	
\end{thebibliography}
\end{document}